\documentclass{amsart}

\usepackage{amssymb,amstext,amsmath,amscd,amsthm,amsfonts,enumerate,graphicx,latexsym,
color,stmaryrd, amsxtra}
\usepackage{mathptmx}
\usepackage[notcite, notref]{}
\usepackage{nicefrac}

\usepackage[usenames,dvipsnames]{pstricks}
\usepackage[mathscr]{eucal}

\usepackage{enumerate,verbatim}
\usepackage[all,2cell,ps]{xy}
\usepackage{mathptmx}
\usepackage[notcite, notref]{}
\usepackage[pagebackref]{hyperref}
\usepackage{nicefrac}
\usepackage{array}
\usepackage{xcolor}
\usepackage{amsthm}
\usepackage[leqno]{amsmath}
\usepackage[pagebackref]{hyperref}

\usepackage{mathptmx}

\usepackage{amsfonts,amsmath,amssymb,amsthm,amscd,amsxtra}
\usepackage{enumerate,verbatim}
\usepackage[all,2cell,ps]{xy}
\usepackage[pagebackref]{hyperref}
\usepackage{todonotes}

\theoremstyle{plain}
\newtheorem{thm}{Theorem}[section]
\newtheorem{prop}[thm]{Proposition}

\theoremstyle{definition}
\newtheorem{lem}[thm]{Lemma}

\newtheorem{chunk}[thm]{}

\newtheorem{rmk}[thm]{Remark}

\numberwithin{equation}{section}

\newcommand{\qis}{\simeq}
\newcommand{\xra}[2][]{\xrightarrow[#1]{\;#2\;}}
\newcommand{\qra}{\xra{\qis}}
\newcommand{\gd}{\mbox{G-dim}\,}
\newcommand{\ZZ}{\mathbb{Z}}
\newcommand{\ac}{\mathcal{A}_C}
\newcommand{\bc}{\mathcal{B}_C}
\newcommand{\mpc}{\mathcal{P}_C}
\newcommand{\mic}{\mathcal{I}_C}

\DeclareMathOperator{\hh}{H}
\DeclareMathOperator{\im}{im}

\def\id{\operatorname{\mathsf{id}}}

\def\G-dim{\operatorname{\mathsf{G-dim}}}
\def\gkd{\operatorname{\mathcal{G}_{\it C}\mathsf{-dim}}}
\def\tate{\operatorname{\widehat{\Tor}}}
\def\gc{\operatorname{\mathsf{G}_{\it C}}}
\def\gmc{\operatorname{\mathcal{G}_{\it C}}}
\def\pd{\operatorname{\mathsf{pd}}}
\def\H{\operatorname{\mathsf{H}}}
\def\pcd{\operatorname{\mathcal{P}_{\it C}\mathsf{-dim}}}
\def\trk{\mathsf{Tr}_{C}}
\def\depth{\operatorname{\mathsf{depth}}}
\def\Ext{\operatorname{\mathsf{Ext}}}
\def\Hom{\operatorname{\mathsf{Hom}}}
\def\Tor{\operatorname{\mathsf{Tor}}}
\def\urltilda{\kern -.15em\lower .7ex\hbox{\~{}}\kern .04em}
\def\urldot{\kern -.10em.\kern -.10em}\def\urlhttp{http\kern -.10em\lower -.1ex
\hbox{:}\kern -.12em\lower 0ex\hbox{/}\kern -.18em\lower 0ex\hbox{/}}

\begin{document}

\title[A study of Tate homology via the approximation theory]{A study of Tate homology via the approximation theory \\with applications to the depth formula}

\author[O. Celikbas]{Olgur Celikbas}
\address{Olgur Celikbas: Department of Mathematics, West Virginia University, Morgantown, WV 26506-6310, U.S.A}
\email{olgur.celikbas@math.wvu.edu}

\author[L. Liang]{Li Liang$^{\ast}$}
\address{Li Liang: 1. Department of Mathematics, Shantou University, Shantou 515063, China \ \ \ 2. School of Mathematics and Physics, Lanzhou Jiaotong University, Lanzhou 730070, China}
\email{lliangnju@gmail.com}
\urladdr{https://sites.google.com/site/lliangnju}
\thanks{$^{\ast}$ Corresponding author.}

\author[A. Sadeghi]{Arash Sadeghi}
\address{Arash Sadeghi: School of Mathematics, Institute for Research in Fundamental Sciences, (IPM), P.O. Box:
19395-5746, Tehran, Iran}
\email{sadeghiarash61@gmail.com}

\author[T. Sharif]{Tirdad Sharif}
\address{Tirdad Sharif: School of Mathematics, Institute for Research in Fundamental Sciences, (IPM), P.O. Box:
	19395-5746, Tehran, Iran}
\email{s.tirdad@gmail.com}

\thanks{L. Liang was partly supported by NSF of China grant 11761045, and NSF of Gansu Province grant 21JR7RA297; A. Sadeghi and T. Sharif were partly supported by a grant from IPM}





\begin{abstract}  In this paper we are concerned with absolute, relative and Tate Tor modules. In the first part of the paper we generalize a result of Avramov and Martsinkovsky by using the Auslander-Buchweitz approximation theory, and obtain a new exact sequence connecting absolute Tor modules with relative and Tate Tor modules. In the second part of the paper we consider a depth equality, called the depth formula, which has been initially introduced by Auslander and developed further by Huneke and Wiegand. As an application of our main result, we generalize a result of Yassemi and give a new sufficient condition implying the depth formula to hold for modules of finite Gorenstein and finite injective dimension.
\end{abstract}

\maketitle

\thispagestyle{empty}

\section{Introduction}
\noindent
Throughout $R$ denotes a commutative noetherian local ring. Moreover, all $R$-modules, unless otherwise stated, are assumed to be finitely generated.

Relative and Tate (co)homology theories for $R$-modules of finite Gorenstein dimension were introduced and studied by Avramov and Martsinkovsky in their beautiful paper \cite{AM}. In particular, an interesting connection between absolute, relative and Tate homology, and cohomology, modules was established in \cite{AM}; see also \cite[Theorem 1]{Ia}. We recall the homology version of the aforementioned result next:

\begin{thm} (Avramov and Martsinkovsky; see \cite[7.1 and 7.4]{AM}) \label{AM} Let $M$ and $N$ be $R$-modules. Assume $M$ has finite Gorenstein dimension $n$ for some positive integer $n$. Then there exists an exact sequence of the form:
$$\begin{CD}
0@>>>\tate_n^{R}(M,N)@>>>\Tor_n^R(M,N)@>>>\mathcal{G}\Tor_n^R(M,N)@>>>&  \\
\cdots@>>>\tate_1^{R}(M,N)@>>>\Tor_1^R(M,N)@>>>\mathcal{G}\Tor_1^R(M,N)@>>>0.
\end{CD}$$
\end{thm}

Our first purpose in this paper is to improve Theorem \ref{AM}, and prove the following generalization:

\begin{thm} \label{thmintro} Let $M$ and $N$ be $R$-modules, and let $C$ be a semidualizing $R$-module. Assume $M$ has finite $\gmc$-dimension $n$ for some positive integer $n$. Assume further $N$ belongs to the Auslander class $\mathcal{A}_C$. Then there is an exact sequence of the form:
	$$\begin{CD}
	0@>>>\tate_n^{\mathcal{G}_C}(M,N)@>>>\Tor^R_n(M,N)@>>>\mathcal{G}_C\Tor^R_n(M,N)@>>>&  \\
	\cdots@>>>\tate_1^{\mathcal{G}_C}(M,N)@>>>\Tor^R_1(M,N)@>>>\mathcal{G}_C\Tor^R_1(M,N)@>>>0.
	\end{CD}$$
\end{thm}

We use the Auslander-Buchweitz approximation theory to prove the above result. The classical Tate Tor modules $\tate_{*}^{R}(M,N)$ are computed by using a complete resolution of $M$, i.e., by using a diagram $T \xra{\tau} P \qra M $, where $P \qra M$ is a projective resolution of $M$, $T$ is a totally acyclic complex of free $R$-modules, and $\tau_i$ is an isomorphism for all $i \gg 0$. Therefore, $\tau$ induces a well-defined map on the Tor modules $\tate_n^{R}(M,N)\to \Tor_n^R(M,N)$ for each $i=1,\ldots, n$. However, we are not aware of a morphism, similar to $\tau$, which induces a map $\tate_i^{\mathcal{G}_C}(M,N)\to \Tor^R_i(M,N)$ for each $i=1,\ldots, n$, so a completely new approach is needed in proving Theorem \ref{thmintro}; we prove this result via maps induced on Tor modules via a $\gmc$-approximation of $M$ introduced by Auslander and Buchweitz \cite{AuBu}.

Let us also draw attention to an important point concerning Theorem \ref{thmintro}: we prove in Proposition \ref{ppcc} that Tate Tor modules $\tate_{*}^{\mathcal{G}_C}(M,N)$ with respect to a semidualizing module $C$ are well-defined as long as $N$ belongs to the Auslander class $\mathcal{A}_C$; see also \ref{ABC}.

We recall the definitions of relative and Tate Tor modules, $\mathcal{G}_C\Tor^R_{*}(M,N)$ and $\tate_{*}^{\mathcal{G}_C}(M,N)$, and record various preliminary results in Section 2. Here let us point out that
Theorem \ref{thmintro} is a generalization of Theorem \ref{AM}; more precisely Theorem \ref{AM} follows immediately from Theorem \ref{thmintro} when the semidualizing module $C$ equals $R$.

\begin{center}
  $\ast \ \ \ast \ \ \ast$
\end{center}

\noindent
Auslander \cite{A} proved that, if $M$ and $N$ be $R$-modules with $\pd_R(M)<\infty$ and $q=\sup\{i\mid\Tor_i^R(M,N)\neq0\}$, then the depth equality $\depth_R(M)+\depth_R(N)=\depth R+\depth_R(\Tor_q^R(M,N))-q$ holds, provided that either $q=0$ or $\depth_R(\Tor_q^R(M,N))\leq1$. Huneke and Wiegand \cite{HW} studied this depth equality over complete intersection rings for the case where $q=0$, and dubbed it ``the depth formula." Following the work of Huneke and Wiegand, the depth formula has been a subject of intense study and various conditions implying the depth formula have been obtained by researchers including Araya and Yoshino \cite{AY}, Foxy \cite{F2} and Iyengar \cite{Iy}; see also \cite{CI, CJ, SSS}.

A result of Yassemi \cite[2.11]{Y} concerning the depth formula, implies that, if $ M$ and $N$ are $R$-modules such that $\G-dim_R(M)<\infty$, $\id_R(N)<\infty$, and $\Tor_i^R(M,N)=0$ for all $i\geq 1$, then the depth formula holds, i.e., one has $\depth_R(M)+\depth_R(N)=\depth R+\depth_R(M\otimes_RN)$. One of the reasons why Yassemi's result is interesting for us is that it uses the finiteness of the injective dimension instead of the projective dimension to establish the depth formula. In section 4, we make use of Theorem \ref{thmintro} and generalize Yassemi's result as follows:

\begin{thm}\label{thmintro2} Let $M$ and $N$ be $R$-modules, and let $C$ be a semidualizing $R$-module. Assume $\gkd_R(M)<\infty$ and $\mathcal{I}_{C}$-$\id_R(N)<\infty$. Set $q=\sup\{i\mid\Tor_i^R(M,N)\neq0\}$. If $q=0$ or $\depth_R(\Tor_q^R(M,N))\leq1$, then it follows that $$\depth_R(M)+\depth_R(N)=\depth R+\depth_R(\Tor_q^R(M,N))-q.$$
\end{thm}
We should note that, under the setup of Theorem \ref{thmintro2}, we have $\Tor_i^R(M,N) \cong \mathcal{G}_C\Tor^R_i(M,N)$ for all $i\geq 0$ so that $\Tor_i^R(M,N)=0$ for all $j\gg 0$, i.e., $q<\infty$; see Theorem \ref{thmintro}, \ref{HD}(iii), and Proposition \ref{llll}. We should also note that Yassemi's result discussed above is an immediate consequence of Theorem \ref{thmintro2} when $q=0$ and $C=R$; see section 2 for the definitions and notations.

Section 4 is devoted to the proof of Theorem \ref{thmintro2}: as the proof of the theorem requires substantial preparation, we first establish several preliminary results, and then give a proof of the theorem at the end of the section; see the paragraph following the proof of Proposition  \ref{llll}. One of these preliminary results we prove is Theorem \ref{tt}: this theorem plays an important role in the proof of Theorem \ref{thmintro2}, and it is a vast generalization of Auslander's classical depth formula result stated above.

In the appendix, we provide motivation for our work concerning section 4 on why we study new conditions that imply the depth formula. More precisely, to faciliate the discussion, we record some remarks and questions, and recall several related applications of Auslander's depth formula from the literature.

\section{Preliminary definitions and properties}
\noindent
In this section we collect various definitions and results needed for our arguments, and provide some background on relative homological algebra. For the unexplained standard terminology, we refer to the reader to \cite{AB, BH}.

\begin{chunk} (\textbf{Semidualizing modules}; see \cite{F, G, V}) \label{sd}
An $R$-module $C$ is called a \emph{semidualizing} module provided that the homothety morphism $R\to\Hom_R(C,C)$ is bijective, and $\Ext^i_R(C,C)=0$ for all $i\geq 1$.

Semidualizing modules were initially studied by Foxby \cite{F}, Vasconcelos \cite{V} and Golod \cite{G}. A dualizing module, if it exists, is an example of a semidualizing module.
\pushQED{\qed}
\qedhere
\popQED	
\end{chunk}

In this paper $C$ denotes a semidualizing module.

\begin{chunk} (\textbf{Auslander transpose}; see \cite{AB, F}) \label{at} Let
$P_1\overset{f}{\rightarrow}P_0\rightarrow M\rightarrow 0$ be a projective presentation of $M$.
The transpose of $M$ with respect to $C$, denoted by $\trk M$, is given by the following exact sequence
\begin{equation}\tag{\ref{at}.1}
0\rightarrow M^{\triangledown}\rightarrow
P_0^{\triangledown}\overset{f^{\triangledown}}{\longrightarrow} P_1^{\triangledown}\rightarrow \trk
M\rightarrow 0,
\end{equation}
where $(-)^{\triangledown}=\Hom_R(-,C)$.
\pushQED{\qed}
\qedhere
\popQED	
\end{chunk}

\begin{chunk}(\textbf{Reflexivity} \cite{F,G}) \label{cr} The module $M$ is said to be \emph{$C$-reflexive} if the canonical map $M\rightarrow M^{\triangledown\triangledown}$ is bijective. Also, $M$ is said to be  \emph{totally $C$-reflexive} if $M$ is $C$-reflexive and $\Ext^i_R(M,C)=0=\Ext^i_R(M^{\triangledown},C)$ for all $i\geq 1$. \pushQED{\qed}
\qedhere
\popQED	
\end{chunk}

We recall the definitions and some basic properties of Auslander and Bass classes:

\begin{chunk} (\textbf{Auslander and Bass classes} \cite{F}) \label{ABC}
\begin{enumerate}[\rm(i)]
\item The {\em Auslander class with respect to} $C$, denoted by $\ac$, consists of all $R$-modules $X$ satisfying:
\begin{enumerate}[\rm(a)]
\item  The natural map $\mu_{X}:X\longrightarrow\Hom_R(C,X\otimes_RC)$ is bijective.
\item $\Tor_i^R(X,C)=0=\Ext^i_R(C,X\otimes_RC)$ for all $i\geq 1$.
\end{enumerate}
\item The {\em Bass class with respect to} $C$, denoted by $\bc$, consists of all $R$-modules $Y$ satisfying:
\begin{enumerate}[\rm(a)]
\item The natural evaluation map $\nu^{Y}: C\otimes_R\Hom_R(C,Y) \longrightarrow Y$ is bijective.
\item $\Tor_i^R(\Hom_R(C,Y),C)=0=\Ext^i_R(C,Y)$ for all $i\geq 1$.
\end{enumerate}
\item If any two modules in a short exact sequence of $R$-modules belong to $\ac$ (respectively, to $\bc$), then does the third one; see \cite[1.3]{F}.
Hence, each $R$-module of finite projective dimension belongs to the Auslander class $\ac$. Moreover, the class $\bc$ contains all modules of finite injective dimension.
\item Assume $R$ is Cohen-Macaulay with a dualizing module $\omega_R$. Then it follows that $M\in\mathcal{A}_{\omega_R}$ if and only if $\gd_R(M)<\infty$; see \cite[Theorem 1]{F1}.
\item Assume $M\in\bc$. Then it follows from the definition that $\Ext^i_R(C,M)=0$ for all $i\geq 1$. Hence \cite[4.1]{AY} shows that $\depth_R(M)=\depth_R(\Hom_R(C,M))$.
\item Assume $M\in\ac$. Then it follows that $\depth_R(M)=\depth_R(M\otimes_RC)$; see  \cite[2.11]{DS}.
\item If $N\in \ac$ is a module, then \cite[5.1]{Sa} extends the result of Auslander and Bridger \cite[2.6]{AB} and yields an exact sequence:
$0\rightarrow\Ext^1_R(\trk M,N\otimes_RC)\rightarrow M\otimes_RN\rightarrow\Hom_R(M^\triangledown,N\otimes_RC)\rightarrow\Ext^2_R(\trk M,N\otimes_RC)\rightarrow0$.
\pushQED{\qed}
\qedhere
\popQED	
\end{enumerate}
\end{chunk}

In order to recall the definitions of Tors with respect to a semidualizing module, we first set up some conventions:

\begin{chunk}\label{cx} In our context, an $R$-complex $X$ is a complex of $R$-modules that is indexed homologically.  For an integer $n$, we define
\begin{enumerate}[\rm(i)]
\item the $n$-{\em fold shift} of $X$ is the complex $\Sigma^n X$ with $(\Sigma^n X)_i = X_{i-n}$, and
$\partial_{i}^{\Sigma^n X} =(-1)^n\partial_{i-n}^X$.
\item the {\em hard truncation} complex $X_{\geq n}$ of $X$ is defined as follows: $(X_{\geq n})_i=0$ for $i<n$, $(X_{\geq n})_i=X_i$ for $i\geq n$, and $\partial^{X_{\geq n}}_{i}=\partial^X_i$ for $i>n$.
\pushQED{\qed}
\qedhere
\popQED	
\end{enumerate}
\end{chunk}

\begin{chunk} (\cite{HJ}) \label{CP} The category $\mpc(R)$ of \emph{$C$-projective} modules is defined as $\{P\otimes_RC : P \text{ is a projective } R \text{-module}\}$. Note that, since $R$ is local and all $R$-modules are assumed to be finitely generated, an $R$-module $M$ is $C$-projective if and only if $M$ is isomorphic to a finite direct sum copies of the semidualizing module $C$.

Similarly, one defines the category of \emph{$C$-injective} modules as $\mic(R)=\{\Hom_R(C,I) : I \text{ is an injective } R \text{-module}\}$. Note that a $C$-injective $R$-module is not necessarily finitely generated. In fact, injective $R$-modules are not necessarily finitely generated as we do not assume the ring $R$ to be Cohen-Macaulay; see \cite[9.6.2 and 9.6.4]{BH}.
\pushQED{\qed}
\qedhere
\popQED	
\end{chunk}

\begin{chunk} (\textbf{Sequences and Resolutions}) \label{PCR}
\begin{enumerate}[\rm(i)]
\item A sequence $\eta$ of $R$-module homomorphisms is called \emph{$\gmc$-proper} if the
induced sequence $\Hom_R(A,\eta)$ is exact for every totally $C$-reflexive $R$-module $A$; see \ref{cr}.
\item A \emph{$\gmc$-resolution} of $M$ is a right acyclic complex $X$ of totally $C$-reflexive $R$-modules such that $\H_0(X) \cong M$.
\item A \emph{$\gmc$-proper resolution} of $M$ is a $\gmc$-resolution $X$ of $M$ such that the augmented resolution $X^+$ is a $\gmc$-proper sequence, i.e.,
the complex
$$\Hom_R(A, X^{+})= \cdots \to \Hom_R(A,X_1) \to \Hom_R(A, X_0) \to \Hom_R(A, M) \to 0$$
is acyclic for each totally $C$-reflexive $R$-module $A$; see \ref{cr}.
\item A \emph{$\mathcal{P}_C$-resolution} of $M$ is a right acyclic complex $X$ of $C$-projective $R$-modules such that $\H_0(X) \cong M$. One defines a \emph{$\mathcal{P}_C$-proper} resolution as in part (iii). It follows from \cite[Lemma 3.1(c)]{STSY} that every $R$-module has a $\mathcal{P}_C$-proper resolution.
\item An \emph{$\mathcal{I}_C$-resolution} of $M$ is a left acyclic complex $X$ of $C$-injective $R$-modules such that $\H^{0}(X) \cong M$.
\item A \emph{complete $\mathcal{P}\mpc$-resolution} is an acyclic complex $X$ of $R$-modules such that
\begin{enumerate}[\rm(a)]
\item $X_i$ is projective for each $i\geq 0$, and $X_i$ is $C$-projective for each $i\leq -1$.
\item $\Hom_R(X,A)$ is an acyclic complex for each $C$-projective $R$-module $A$. \pushQED{\qed}
\qedhere
\popQED	
\end{enumerate}
\end{enumerate}
\end{chunk}

The classical Gorenstein dimension \cite{AB} was extended to $\gmc$-dimension by Foxby in \cite{F} and Golod in \cite{G}. We recall its definition along with some other homological dimensions with respect to the semidualizing module $C$.

\begin{chunk} (\textbf{Homological dimensions}) \label{HD}
\begin{enumerate}[\rm(i)]
\item The $\gmc$-dimension of $M$, denoted by $\gkd_R(M)$, is said to be finite if $M$ has a $\gmc$-resolution of finite length. We set $\gkd_R(0)=-\infty$. If $M$ is not zero, it follows $\gkd_R(M)=0$ if and only if $M$ is totally $C$-reflexive. If $\gkd_R(M)=n<\infty$, then $M$ has a $\gmc$-proper resolution $X$, even one such that $X_i$ is $C$-projective for each $i=1, \ldots, n$ and $X_i=0$ for each $i\geq n+1$; see \cite[2.3]{ATY}.
\item The $\mpc$-dimension of $M$, denoted by $\pcd_R(M)$, is said to be finite if $M$ has a $\mathcal{P}_C$-resolution of finite length; see \cite[2.10(a)]{TW}. Note that, if $\pcd_R(M)<\infty$, then $M\in\bc$; see \cite[2.9(a)]{TW}. Moreover, $\pcd_R(M)=\pd_R(\Hom_R(C,M))$; see \cite[2.11]{TW}.
\item The $\mathcal{I}_C$-injective dimension of $M$, denoted by $\mathcal{I}_{C}$-$\id_R(M)$, is said to be finite if $M$ has an $\mathcal{I}_C$-resolution of finite length; see \cite[2.10(b)]{TW}. Note that, if $\mathcal{I}_{C}$-$\id_R(M)<\infty$, then $M\in\mathcal{A}_C$; see \cite[2.9(b)]{TW}. Moreover, $\mathcal{I}_{C}$-$\id_R(M)=\id_R(M\otimes_RC)$ \cite[2.11]{TW}.
\pushQED{\qed}
\qedhere
\popQED	
\end{enumerate}
\end{chunk}

In the following, we record some facts about $\gmc$-dimension:

\begin{chunk} \label{AA} The semidualizing module $C$ is dualizing if and only if $\gkd_R(Y)<\infty$ for all $R$-modules $Y$; see \cite[1.3]{Ge}. Furthermore, by \cite{G}, it follows that:
\begin{enumerate}[\rm(i)]
\item $\gkd_R(M)=0$ if and only if $\gkd_R(\trk M)=0$.
\item If $\gkd_R(M)<\infty$, then $\gkd_R(M)=\depth(R)-\depth_R(M)$.
\item If $\gkd_R(M)<\infty$, then $\gkd_R(M)=\sup\{i\geq0\mid\Ext^i_R(M,C)\neq0\}$.
\item $\gkd_R(M)=0$ if and only if $\Ext^i_R(M,C)=0=\Ext^i_R(\trk M,C)$ for all $i\geq 1$.
\pushQED{\qed}
\qedhere
\popQED	
\end{enumerate}
\end{chunk}

We are now ready to record the definitions of several types of Tor.

\begin{chunk} (\textbf{$\gmc$-Relative homology}; see \cite[2.4]{HH}) \label{RTS}
Let $M$ be an $R$-module of finite $\gmc$-dimension (in this case $M$ has a $\gmc$-proper resolution; see \ref{HD}(i)). For an $R$-module $N$ and an integer $i$, we define the \emph{$\gmc$-relative Tor} as $$\mathcal{G}_C\Tor^R_i(M,N)=\hh_i(X\otimes_R N),$$
where $X$ is a $\gmc$-proper resolution of $M$.

It is known that this construction is well-defined, and also functorial in $M$ and $N$. Notice, it follows from the definition that $\mathcal{G}_C\Tor^R_0(M,N)=M\otimes_R N$. \pushQED{\qed}
\qedhere
\popQED	
\end{chunk}

\begin{chunk} (\textbf{$\mathcal{P}_C$-Relative homology}; see \cite[3.10 and 4.3]{STSY}) \label{TPC} Let $M$ be an $R$-module. For an $R$-module $N$ and an integer $i$, we define the \emph{$\mathcal{P}_C$-relative Tor} as
$$\Tor_i^{\mathcal{P}_C}(M,N)=\hh_i(X \otimes_R N),$$
where $X$ is a $\mathcal{P}_C$-proper resolution of $M$; see \ref{PCR}(iv). We note that:
\begin{enumerate}[\rm(i)]
\item There are natural isomorphisms $\Tor_i^{\mathcal{P}_C}(M,N)\cong\Tor_i^R(\Hom_R(C,M),N\otimes_RC)$ for each $i\geq 0$.
\item If $M\in\bc$ and $N\in\ac$, it follows that $\Tor_i^{\mathcal{P}_C}(M,N)\cong\Tor_i^R(M,N) $ for each $i\geq 0$.
\pushQED{\qed}
\qedhere
\popQED	
\end{enumerate}
\end{chunk}

\begin{chunk} (\textbf{Tate Tor with respect to a semidualizing module}) \label{TTS}   Assume $\gkd_R(M)<\infty$.

Let $\gkd_R(M)=n$ and let ${\bf P}\to M$ be a projective resolution of $M$. Then the $nth$ syzygy module $\Omega^nM$ of $M$ is totally $C$-reflexive. Let ${\bf Q}\to(\Omega^nM)^{\triangledown}\to0$ be a projective resolution of $(\Omega^nM)^{\triangledown}$ and consider the complex  dualized with respect to $C$, i.e., $0\to(\Omega^nM)^{\triangledown\triangledown}\to {\bf Q^{\triangledown}}$.
Note that, by the defining properties of totally $C$-reflexive modules, this complex is exact and $\Omega^nM\cong(\Omega^nM)^{\triangledown\triangledown}$.

Hence there is an acyclic complex $0 \to \Omega^nM \to T_{n-1} \to T_{n-2} \to \cdots$ with the following properties:
\begin{enumerate}[\rm(1)]
\item Each $T_i$ is $C$-projective.
\item For each $C$-projective $R$-module $X$, the complex
$\cdots  \to \Hom_R(T_{n-1},X) \to \Hom_R(\Omega^nM,X) \to 0$ is acyclic.
\end{enumerate}
Splicing the complex $0 \to \Omega^nM \to T_{n-1} \to T_{n-2} \to \cdots$ with the projective resolution of $M$, i.e., with the acyclic complex $\cdots \to P_{n+1} \to P_{n} \to \Omega^nM \to 0$, we obtain an acyclic complex of $R$-modules
$${\bf T}: \cdots\longrightarrow T_{n+1}\overset{\partial_{n+1}}{\longrightarrow} T_{n}\overset{\partial_{n}}{\longrightarrow} T_{n-1}\overset{\partial_{n-1}}{\longrightarrow} T_{n-2}\longrightarrow\cdots$$
in which $\im \partial_n =\Omega^n(M)$, $T_i$ is projective for each $i\geq n$, and $T_i$ is $C$-projective for each $i<n$. Then it follows that $\mathsf{\Sigma}^{-n}{\bf T}$ is a complete $\mathcal{P}\mathcal{P}_C$-resolution, and ${\bf T}_{\geq n}={\bf P}_{\geq n}$.

For an $R$-module $N$ that belongs to $\ac$ and for an integer $i$, we define the \emph{Tate Tor} as:
\begin{equation}\tag{\ref{TTS}.1}
\tate_i^{\mathcal{G}_C}(M,N)=\hh_i({\bf T}\otimes_RN).
\end{equation}
Note that, by construction, there are isomorphisms
\begin{equation}\tag{\ref{TTS}.2}
\tate_i^{\mathcal{G}_C}(M,N)\cong\Tor_i^R (M,N) \text{ for all } i>\gkd_R(M). \pushQED{\qed}
\qedhere
\popQED	
\end{equation}
\end{chunk}

We finish this section by proving in Proposition \ref{ppcc} that the Tor modules $\tate_{*}^{\mathcal{G}_C}(M,N)$, defined as in \ref{TTS}, are independent of the choice of the acyclic complex $T$. It is worth noting, for this property, we need to assume $N\in\ac$. We proceed and record a result from \cite{TW} which will be used later. In passing, we recall that $C$-injective $R$-modules are not necessarily finitely generated; see \ref{CP}.

\begin{chunk} (\cite[2.3(b)]{TW}) \label{noted} Assume $M\in\ac$. Then there exists an exact sequence $0\to M\to X\to Y\to 0$,  where $X$ is an $C$-injective module and $Y\in\ac$. \pushQED{\qed}
\qedhere
\popQED	
\end{chunk}

\begin{lem}\label{ppc} $T\otimes_RJ$ is an acyclic complex if $T$ is a complete $\mathcal{P}\mathcal{P}_C$-resolution and $J$ is an $C$-injective $R$-module.
\end{lem}

\begin{proof} Note, as $J$ is $C$-injective, it follows from \cite[4.1]{SSW} that $J^+=\Hom_{\mathbb{Z}}(J,\mathbb{Q}/\mathbb{Z})$ is $C$-flat. Therefore $J^+$ admits a bounded $\mathcal{P}_C$-resolution  because every flat module has finite projective dimension; see \cite[Prop. 6]{Jen}. Hence, by \cite[2.7]{W}, we see that the complex $\Hom_R(T,J^+)$ is acyclic; this implies $T\otimes_RJ$ is acyclic.
\end{proof}

\begin{prop} \label{ppcc}Let $M$ and $N$ be $R$-modules such that  $\gkd_R(M)=n<\infty$ and $N\in \ac$. Let $(T,P)$ and $(T',P')$ be a pair of complexes defined as in \ref{TTS}. Then, for each $i\in\ZZ$, there is an isomorphism $\hh_{i}(T\otimes_RN) \cong \hh_{i}(T'\otimes_RN)$. Therefore, $\tate_{*}^{\mathcal{G}_C}(M,N)$ is independent of the choice of the acyclic complex $T$ constructed in \ref{TTS}.
\end{prop}

\begin{proof} Note, since $N\in\ac$, it follows from \ref{noted} that there is an exact sequence of $R$-modules $0 \to N \to J \to K \to 0$, where $J$ is $C$-injective. One has $K\in\ac$, so that $\Tor_i^R(Q,K)=0$ for all $C$-projective $R$-modules $Q$ and for all $i\geq 1$. Thus, we have an exact sequence of complexes $0 \to T\otimes_RN \to T\otimes_RJ \to T\otimes_RK \to 0$. Note, by Lemma \ref{ppc}, the complex $T\otimes_RJ$ is acyclic. Therefore, it follows that $\hh_i(T\otimes_RN) \cong \hh_{i+1}(T\otimes_RK)$ for all $i\in\ZZ$. Similarly, we see that $\hh_i(T'\otimes_RN) \cong \hh_{i+1}(T'\otimes_RK)$ for all $i\in\ZZ$.

Recall that $T_{\geq n}=P_{\geq n}$ and $T'_{\geq n}=P'_{\geq n}$; see \ref{TTS}. Hence, for all $i\geq n+1$, we have $\hh_{i}(T\otimes_RN) \cong \hh_{i}(T'\otimes_RN)$ and $\hh_{i}(T\otimes_RK) \cong \hh_{i}(T'\otimes_RK)$. This yields the isomorphism $\hh_{i}(T\otimes_RN) \cong \hh_{i}(T'\otimes_RN)$ for each $i\in\ZZ$.
\end{proof}

\section{Proof of the main result}
\noindent
In this section we give a proof of our main theorem, namely Theorem \ref{thmintro}. First we prove and record several results that are needed for our argument.

\begin{rmk}\label{l}
Let $M$ and $N$ be $R$-modules.
\begin{enumerate}[\rm(i)]
\item If $M$ is totally $C$-reflexive and $\pcd_R(N)<\infty$, then it follows $\Ext^i_R(M,N)=0$ for all $i\geq 1$; see \cite[3.7]{SSW1}.
\item Assume $\pcd_R(M)<\infty$. Then it follows from part (i) that each $\mathcal{P}_C$-resolution of $M$ is also a proper $\mathcal{G}_C$-resolution. Therefore,
$\mathcal{G}_C\Tor_i(M,N)$ is well-defined for each integer $i$; see \ref{HD}(i) and \ref{RTS}. Also, we have that $\Tor_i^{\mathcal{P}_C}(M,N) \cong \mathcal{G}_C\Tor_i(M,N)$ for each $i\geq 0$; see \ref{TPC}.
\item If $\pcd_R(M)<\infty$ and $N\in\mathcal{A}_C$, then we have $\Tor_i^R(M,N) \cong \Tor_i^{\mathcal{P}_C}(M,N) \cong \mathcal{G}_C\Tor_i(M,N) $ for each $i\geq0$; note that this isomorphism follows from part (ii) in view of \ref{HD}(ii) and \ref{TPC}(ii).
\end{enumerate}
\end{rmk}

\begin{lem}\label{ll} Let $0\rightarrow Q\rightarrow X\rightarrow M\rightarrow0$ be a $\gmc$-proper exact sequence of $R$-modules of finite $\gmc$-dimension, and let $N$ be an $R$-module. Then, for each integer $i$, there is an exact sequence of the form:
	$$\cdots\rightarrow\mathcal{G}_C\Tor_i(Q,N)\rightarrow\mathcal{G}_C\Tor_i(X,N)\rightarrow\mathcal{G}_C\Tor_i(M,N)\rightarrow \mathcal{G}_C\Tor_{i-1}(Q,N)\rightarrow \cdots.$$
\end{lem}

\begin{proof} The claim is a consequence of \cite[8.2.3]{EJ}.
\end{proof}

\begin{lem}\label{l33} Let $0\rightarrow N'\rightarrow N\rightarrow N''\rightarrow0$ be an exact sequence of $R$-modules, where $N'$, $N$ and $N''$ belong to $\ac$, and let $M$ be an $R$-module such that $\gkd_R(M)<\infty$. Then, for each integer $i$, there is an exact sequence of the form:
$$\cdots\longrightarrow\tate_i^{\mathcal{G}_C}(M,N)\longrightarrow\tate_i^{\mathcal{G}_C}(M,N'')
\longrightarrow\tate_{i-1}^{\mathcal{G}_C}(M,N')\longrightarrow\cdots$$
\end{lem}

\begin{proof}
Let $\gkd_R M=n$. Consider a projective resolution $P$ of $M$ and an acyclic complex $T$  such that $T_{\geq n}=P_{\geq n}$ and $\mathsf{\Sigma}^{-n}T$ is a complete $\mathcal{P}\mathcal{P}_C$-resolution; see \ref{TTS}. As $N''\in\ac$, the exact sequence $0\rightarrow N'\rightarrow N\rightarrow N''\rightarrow0$ yields the exact sequence of complexes as $0 \to T\otimes_RN' \to T\otimes_RN \to T\otimes_RN'' \to 0$; see \ref{ABC}(i)(b). This exact sequence of complexes yields the required long exact sequence.
\end{proof}

\begin{lem}\label{l3} Let $0\rightarrow M'\rightarrow M\rightarrow M''\rightarrow0$ be an exact sequence of $R$-modules of finite $\gmc$-dimension, and let $N$ be an $R$-module that belongs to $\ac$. Then, for each integer $i$, there is an exact sequence of the form:
$$\cdots \longrightarrow \tate_i^{\mathcal{G}_C}(M',N) \longrightarrow\tate_i^{\mathcal{G}_C}(M,N)\longrightarrow\tate_i^{\mathcal{G}_C}(M'',N)
\longrightarrow\tate_{i-1}^{\mathcal{G}_C}(M',N)\longrightarrow\cdots.$$
\end{lem}

\begin{proof} Set $\max\{\gkd_R M',\gkd_R M''\}=n$. Let $P'\to M'$ and $P''\to M''$ be projective resolutions of $M'$ and $M''$, respectively. Then, by the classical Horseshoe lemma, there is a projective resolution $P \to M$ of $M$ such that the sequence $0 \to P' \to P \to P'' \to 0$ is degree-wise split exact.

Let $T'$ and $T''$ be acyclic complexes such that $T'_{\geq n}= P'_{\geq n}$, $T''_{\geq n}= P''_{\geq n}$, and $\mathsf{\Sigma}^{-n}T'$ and $\mathsf{\Sigma}^{-n}T''$ are complete $\mathcal{P}\mathcal{P}_C$-resolutions; see \ref{TTS}. Now, by using a similar construction as in \cite[5.5]{AM} (see also \cite[3.9]{SaSW}), one gets an acyclic complex $T$ such that $T_{\geq n}=P_{\geq n}$, $\mathsf{\Sigma}^{-n}T$ is a complete $\mathcal{P}\mathcal{P}_C$-resolution, and there is a degree-wise split exact sequence $0 \to T' \to T \to T'' \to 0$ (we note that, for this argument, there is no requirement to construct a morphism of complexes from $T$ to $P$.) This yields the exact sequence $0 \to T'\otimes_RN \to T\otimes_RN \to T''\otimes_RN \to 0$. Now the desired long exact sequence follows due to the facts $\tate_i^{\mathcal{G}_C}(M',N)\cong\hh_i(T'\otimes_RN)$, $\tate_i^{\mathcal{G}_C}(M,N)\cong\hh_i(T\otimes_RN)$ and $\tate_i^{\mathcal{G}_C}(M'',N)\cong\hh_i(T''\otimes_RN)$.
\end{proof}

\begin{prop}\label{l4} Let $M$ be an $R$-module such that $\pcd_R M<\infty$. If $N$ is an $R$-module such that $N\in \ac$, then the module $\tate_i^{\mathcal{G}_C}(M,N)$ is defined and vanishes for each integer $i$.
\end{prop}

\begin{proof} It is clear from \ref{HD} that $\gkd_R M<\infty$. Hence, if $Y$ is an $R$-module such that $Y\in \ac$, then $\tate_i^{\mathcal{G}_C}(M,Y)$ is defined for each integer $i$; see (\ref{TTS}.1).

Set $\pcd_R M=n$. Note that $M\in\bc$; see \ref{HD}(ii). Moreover, by \cite[2.16]{W}, it follows that $\gkd_R M=n$. Hence, for each $R$-module $Y$ with $Y\in\ac$, we have:
\begin{equation}\tag{\ref{l4}.1}
\tate_i^{\mathcal{G}_C}(M,Y)\cong\Tor_i^R (M,Y)\cong\Tor_i^{\mathcal{P}_C}(M,Y)=0 \text{ for each } i>n.
\end{equation}
Here, in (\ref{l4}.1), the first isomorphism follows from (\ref{TTS}.2), the second isomorphism holds by Remark \ref{l}(iii), and the equality is due to the assumption $\pcd_R M=n$.
Next we proceed to prove a claim.

Claim: If $J$ is a $C$-injective $R$-module, then $\tate_i^{\mathcal{G}_C}(M,J)$ is defined and vanishes for each integer $i$.\\
Proof of the claim: Note that $J\in \ac$; see \ref{HD}(iii). Therefore, $\tate_i^{\mathcal{G}_C}(M,J)$ is defined for each integer $i$; see (\ref{TTS}.1). Let $P\to M$ be a projective resolution of $M$, and let $T$ be an acyclic complex such that $T_{\geq n}= P_{\geq n}$ and $\mathsf{\Sigma}^{-n}T$ is a complete $\mathcal{P}\mathcal{P}_C$-resolution; see \ref{TTS}. Then Lemma \ref{ppc} implies that $\hh_i(T\otimes_RJ)=0$ for each integer $i$. As $\tate_i^{\mathcal{G}_C}(M,J)\cong \hh_i(T\otimes_RJ)$ for each integer $i$, the claim follows.

Now let $N$ be an $R$-module such that $N\in \ac$. Fix an integer $j$ and consider $\tate_j^{\mathcal{G}_C}(M,N)$. If $j>n$, then $\tate_j^{\mathcal{G}_C}(M,N)$ vanishes by (\ref{l4}.1). Hence we may assume $j\leq n$.

It follows from \ref{noted} that there is an exact sequence $0 \to N \to J_1 \to Y_1 \to 0$, where $J_1$ is $C$-injective and $Y_1\in\ac$. It follows by the claim and Lemma \ref{l33} that:
\begin{equation}\notag{}
\tate_i^{\mathcal{G}_C}(M,N) \cong \tate_{i+1}^{\mathcal{G}_C}(M,Y_1) \text{ for each integer } i.
\end{equation}

As $Y_1\in\ac$, we make use of  \ref{noted} with $Y_1$, and obtain an exact sequence $0 \to Y_1 \to J_2 \to Y_2 \to 0$, where $J_2$ is $C$-injective and $Y_2\in\ac$. It follows by the claim and Lemma \ref{l33} that:
\begin{equation}\notag{}
\tate_i^{\mathcal{G}_C}(M, Y_1) \cong \tate_{i+1}^{\mathcal{G}_C}(M,Y_2) \text{ for each integer } i.
\end{equation}

Therefore, for each integer $i$, it follows that $\tate_i^{\mathcal{G}_C}(M,N) \cong \tate_{i+2}^{\mathcal{G}_C}(M,Y_2)$. We continue in this way; for each $r\geq 1$, we obtain an $R$-module $Y_r$ such that $Y_r \in \ac$ and the following isomorphism holds:
\begin{equation}\tag{\ref{l4}.2}
\tate_i^{\mathcal{G}_C}(M, N) \cong \tate_{i+r}^{\mathcal{G}_C}(M,Y_r) \text{ for each integer } i.
\end{equation}
Now letting $r=n-j+1$, (\ref{l4}.2) yields $\tate_j^{\mathcal{G}_C}(M, N) \cong \tate_{n+1}^{\mathcal{G}_C}(M,Y_{n-j+1})$, where $\tate_{n+1}^{\mathcal{G}_C}(M,Y_{n-j+1})=0$ due to (\ref{l4}.1). This proves $\tate_i^{\mathcal{G}_C}(M,N)=0$ for each $i\leq n$, and completes the proof of the proposition.
\end{proof}

\begin{rmk}\label{rm} (\cite[1.1]{AuBu}) Let $M$ be an $R$-module of finite $\mathcal{G}_C$-dimension.
\begin{enumerate}[\rm(i)]
\item There is an exact sequence of $R$-modules of the form $0 \to Y \to X \to M \to 0$, where $X$ is totally $C$-reflexive and $Y$ has finite $\mathcal{P}_C$-dimension. Such a short exact sequence is called a \emph{$\mathcal{G}_C$-approximation} of $M$.
\item There is an exact sequence of $R$-modules of the form $0 \to M \to X \to Y \to 0$, where $Y$ is totally $C$-reflexive and $\gkd_R(M)=\mathcal{P}_C$-$\pd_R(X)$. Such a short exact sequence is called a \emph{$\mathcal{G}_C$-hull} of $M$.\pushQED{\qed}
\qedhere
\popQED	
\end{enumerate}
\end{rmk}

We are now ready to prove our main result:

\begin{proof}[Proof of Theorem \ref{thmintro}] We recall that $\gkd_R(M)=n$, and consider a $\gmc$-approximation of $M$, i.e., a short exact sequence of $R$-modules
\begin{equation}\tag{\ref{thmintro}.1}
0\to Y\to X\to M \to 0,
\end{equation}
where $X$ is totally $C$-reflexive and $\pcd_R(Y)<\infty$; see Remark \ref{rm}(i).

Note, since $X$ is totally $C$-reflexive, we have $\mathcal{G}_C\Tor_i(X,N)=0$ for all $i\geq 1$; see \ref{RTS}. Hence, by Remark \ref{l}(i), the exact sequence in (\ref{thmintro}.1) is $\gmc$-proper. Therefore, Lemma \ref{ll} induces the following exact sequence:
\begin{equation}\tag{\ref{thmintro}.2}
0\to\mathcal{G}_C\Tor^R_1(M,N) \to Y\otimes_RN\to X\otimes_RN\to M\otimes_R N\to0.
\end{equation}

Moreover, for each $i\geq 2$, the following isomorphisms hold:
\begin{equation}\tag{\ref{thmintro}.3}
\mathcal{G}_C\Tor^R_i(M,N)\cong\mathcal{G}_C\Tor^R_{i-1}(Y,N)\cong\Tor_{i-1}^R(Y,N).
\end{equation}
Here, in (\ref{thmintro}.3), the first isomorphism follows from (\ref{thmintro}.1) and Lemma \ref{ll}. Moreover, as $N\in \ac$ and $\pcd_R(Y)<\infty$, the second isomorphism in (\ref{thmintro}.3) is due to Remark \ref{l}(iii).

We have that $\tate_{*}^{\mathcal{G}_C}(Y,N)=0$; see Proposition \ref{l4}. So, for each integer $i$, it follows from Lemma \ref{l3} that there is an exact sequence of the form:
\begin{equation}\tag{\ref{thmintro}.4}
0=\tate_i^{\mathcal{G}_C}(Y,N) \longrightarrow \tate_i^{\mathcal{G}_C}(X,N)
\underset{\cong}{\overset{f_i}{\longrightarrow}} \tate_{i}^{\mathcal{G}_C}(M,N) \longrightarrow \tate_{i-1}^{\mathcal{G}_C}(Y,N)=0.
\end{equation}

The construction given in \ref{TTS} yields the following isomorphism for each $i>\gkd_R X=0$:
\begin{equation}\tag{\ref{thmintro}.5}
\tate_i^{\mathcal{G}_C}(X,N) \underset{\cong}{\overset{g_i}{\longrightarrow}} \Tor_i^R(X,N).
\end{equation}

The short exact sequence in (\ref{thmintro}.1), along with \ref{AA}(ii), imply that $\gkd_R(Y)=\pcd_R(Y)=n-1$. Thus, $ \Tor_i^{\mathcal{P}_C}(Y,N) \cong \Tor_i^R(Y,N)=0$ for all $i\geq n$; see \ref{l}(iii).
Now, by applying $-\otimes_RN$ to the short exact sequence in (\ref{thmintro}.1), we get the
following long exact sequence:
\begin{equation}\tag{\ref{thmintro}.6}\\
\begin{CD}
0@>>>\Tor_n^R(X,N)@>\phi_n>>\Tor^R_n(M,N)@>\delta_n>>\Tor_{n-1}^R(Y,N)@>\beta{n-1}>>&\\
\cdots@>>>\Tor_1^R(Y,N)@>\beta_1>>\Tor_1^R(X,N)@>\phi_1>>\Tor_1^R(M,N)@>\delta_1>> \\
&&Y\otimes_RN@>>> X\otimes_RN @>>>M\otimes_R N@>>>0.
\end{CD}
\end{equation}

We make use of the maps from (\ref{thmintro}.4) and (\ref{thmintro}.5), and set $\psi_i=\phi_{i} \circ g_i \circ f_i^{-1}$ for each $i\geq 1$. Thus we obtain the following commutative diagram:
$$\begin{CD}
\tate_i^{\mathcal{G}_C}(X,N)@>f_i>>\tate_i^{\mathcal{G}_C}(M,N)&\\
 @VV{g_i}V @VV{\psi_i}V \\
\Tor_i(X,N)@>{\phi_i}>>\Tor_i(M,N)
\end{CD}$$
Next, for each $i\geq 2$, we set $\alpha_i:\Tor_{i-1}^R(Y,N)\overset{\cong}{\to}\mathcal{G}_C\Tor_i^R(M,N)$ to be the isomorphism obtained from (\ref{thmintro}.3), and $\xi_i=\alpha_i\circ\delta_i:\Tor_i^R(M,N)\to\mathcal{G}_C\Tor_i^R(M,N)$. Moreover, we define  $\xi_1$ as the composition of the following maps:
\begin{equation}\tag{\ref{thmintro}.7}
\Tor_1^R(M,N) \twoheadrightarrow \im(\delta_1)=\ker(Y\otimes_RN\to X\otimes_RN) \cong \mathcal{G}_C\Tor_1(M,N).
\end{equation}
Here, in (\ref{thmintro}.7), the arrow and the first equality follow from (\ref{thmintro}.6), while the isomorphism is due to (\ref{thmintro}.2).

Finally, for each $i\geq 1$, we define the map $\delta'_i$ to be the composition of the following maps:
\begin{equation}\tag{\ref{thmintro}.8}
\begin{CD}
\mathcal{G}_C\Tor^R_{i+1}(M,N) @>\alpha^{-1}_{i+1}>>\Tor^R_i(Y,N) \overset{\beta_i}{\longrightarrow} \Tor^R_{i}(X,N) \overset{g^{-1}_i}{\longrightarrow} \tate_i^{\mathcal{G}_C}(X,N)  \overset{f_i}{\longrightarrow} \tate_i^{\mathcal{G}_C}(M,N).
\end{CD}
\end{equation}

Now, it follows that the following sequence
\begin{equation}\tag{\ref{thmintro}.9}
\cdots\to\mathcal{G}_C\Tor^R_{i+1}(M,N)\overset{\delta'_i}{\longrightarrow}\tate_i^{\mathcal{G}_C}(M,N)\overset{\psi_i}{\longrightarrow}\Tor_i^R(M,N)\overset{\xi_i}{\longrightarrow}\mathcal{G}_C\Tor_i(M,N)\to\cdots
\end{equation}
is exact. As the sequence obtained in (\ref{thmintro}.9) is the required one, the conclusion of the theorem follows.
\end{proof}

\begin{rmk} It should be noted that Sather-Wagstaff, Sharif and White \cite{SaSW} studied Tate cohomology of objects in abelian categories. Their work generalizes the cohomological version of Theorem \ref{AM} for modules $M$ of finite $\gmc$-dimension which belong to the Bass class $\bc$; see also \cite[4.4]{DZLC}.
\end{rmk}

\section{On generalizations of the depth formula}
\noindent
Our main purpose in this section is to give an application of Theorem \ref{thmintro} and prove Theorem \ref{thmintro2}, which generalizes an interesting result of Yassemi discussed in the introduction; see also \cite[2.11]{Y}.  Along the way, we also generalize several related results from the literature. For example, Theorem \ref{tt} is a vast generalization of Auslander's depth formula, which we recall here:

\begin{thm} (Auslander \cite[1.2]{A}) \label{audepth} Let $M$ and $N$ be $R$-modules and let $q=\sup\{i\mid\Tor_i^R(M,N)\neq0\}$. Assume $\pd_R(M)<\infty$. Then, if $q=0$ or $\depth_R(\Tor_q^R(M,N))\leq1$, then the depth formula holds, i.e., $$\depth_R(M)+\depth_R(N)=\depth R+\depth_R(\Tor_q^R(M,N))-q.$$
\end{thm}
The proof of Theorem \ref{thmintro2}, which is given at the end of this section, relies upon Theorem \ref{thmintro}, Proposition \ref{llll} and Theorem \ref{tt}; see the paragraph following the proof of Proposition  \ref{llll}. We start by proving the next proposition which generalizes a result of Avramov and Buchweitz; see \cite[4.4.7]{AvBu}:

\begin{prop}\label{lll} Let $M$ and $N$ be $R$-modules such that $M$ is totally $C$-reflexive and $N\in\mathcal{A}_C$. Then, for each $i\leq0$, the following isomorphism holds:
$$\tate_i^{\mathcal{G}_C}(M,N)\cong\Ext^{-i+1}_R(\trk M,N\otimes_RC).$$
\end{prop}

\begin{proof} We set $(-)^{\triangledown}=\Hom_R(-,C)$, and consider the projective resolutions ${\bf P}\to M \to 0$ and ${\bf Q}\to M^{\triangledown} \to 0$, as well as the acyclic complex ${\bf T}=\cdots\to P_1\to P_0\to Q^\triangledown_0\to Q^\triangledown_1\to\cdots$ constructed in \ref{TTS} (recall $\gkd_R(M)=0$). Then there exist a totally $C$-reflexive $R$-module $M_1$ and a short exact sequence of $R$-modules of the form:
\begin{equation}\tag{\ref{lll}.1}
0\to M \to Q^\triangledown_0\to M_1\to0.
\end{equation}

Note that $\Ext^1_R(M_1,C)=0$. Therefore, by applying the functors $(-)^{\triangledown}$ and $\Hom_R(-,N\otimes_RC)$ to (\ref{lll}.1), respectively, we obtain the following commutative diagram:
$$\begin{CD}
M\otimes_RN@>{\psi}>> Q^\triangledown_0\otimes_RN&\\
@V{\eta}VV @V{\phi}VV \\
\Hom_R(M^\triangledown,N\otimes_RC)@>\alpha>>\Hom_R(Q^{\triangledown\triangledown}_0, N\otimes_RC)
\vspace{0.05in}
\end{CD}$$
Here, in the diagram, $\alpha$ is injective and the maps $\eta$ and $\phi$ are natural; see \ref{ABC}(vii). Notice, since $Q_0$ is projective, $\phi$ is an isomorphism. Thus, we have that $\ker(\psi)=\ker(\eta)$.

Next, by tensoring (\ref{lll}.1) with $N$ and noting that $\Tor_1^R(Q^\triangledown_0,N)=0$ (as $N\in\mathcal{A}_C$), we conclude that the following sequence is exact:
\begin{equation}\tag{\ref{lll}.2}
0\to\Tor_1^R(M_1,N)\to M\otimes_RN \stackrel{\psi}
     {\longrightarrow} Q^\triangledown_0\otimes_RN\to M_1\otimes_RN\to0.
\end{equation}
It now follows from \ref{ABC}(vii) and (\ref{lll}.2) that:
\begin{equation}\tag{\ref{lll}.3}
\Tor_1^R(M_1,N)\cong\ker\psi=\ker\eta\cong\Ext^1_R(\trk M,N\otimes_RC).
\end{equation}

On the other hand, for each $i\geq 1$, we have:
\begin{equation}\tag{\ref{lll}.4}
\Tor_i^R(M_1,N)\cong\tate_i^{\mathcal{G}_C}(M_1,N)\cong\tate_{i-1}^{\mathcal{G}_C}(M,N),
\end{equation}
The first isomorphism of (\ref{lll}.4) follows since $M_1$ is totally $C$-reflexive. Furthermore, the second isomorphism of (\ref{lll}.4) follows from (\ref{lll}.1) in view of Lemma \ref{l3} since $N \in \ac$.

We now see from (\ref{lll}.3) and (\ref{lll}.4) that the conclusion of the proposition holds for $i=0$, i.e.,
\begin{equation}\tag{\ref{lll}.5}
\tate_{0}^{\mathcal{G}_C}(M,N)\cong\Ext^1_R(\trk M,N\otimes_RC).	
\end{equation}

We know $\Ext^1_R(M_1,C)=0$. So we use \cite[2.2]{DS} with (\ref{lll}.1) and deduce that the following sequence is exact:
\begin{equation}\tag{\ref{lll}.6}
0\to\trk M_1\to (\trk C)^{\oplus n} \to\trk M \to 0.
\end{equation}

It follows from \ref{ABC}(vii) that $\Ext^i_R(\trk C, N\otimes_RC)=0$ for $i=1,2$. Thus, (\ref{lll}.6) yields the isomorphism $\Ext^2_R(\trk M,N\otimes_RC)\cong\Ext^1_R(\trk M_1,N\otimes_RC)$. As $M_1$ is totally $C$-reflexive, the isomorphism obtained in (\ref{lll}.5) also holds for the pair $(M_1, N)$. In other words, we have that $\tate_{0}^{\mathcal{G}_C}(M_1,N)\cong\Ext^1_R(\trk M_1,N\otimes_RC)$. So, in view of (\ref{lll}.4), we conclude that the following isomorphisms hold:
\begin{equation}\tag{\ref{lll}.7}
\Ext^2_R(\trk M,N\otimes_RC)\cong\Ext^1_R(\trk M_1,N\otimes_RC)\cong\tate_{0}^{\mathcal{G}_C}(M_1,N)\cong\tate_{-1}^{\mathcal{G}_C}(M,N).
\end{equation}
As the isomorphisms in (\ref{lll}.7) establish the conclusion of the proposition for the case where $i=-1$, we proceed to handle the case where $i\leq -2$.

Notice, as each $Q_i$ in the resolution ${\bf T}$ is projective, the tensor evaluation $(Q_i)^{\triangledown}\otimes_RN\rightarrow\Hom_R(Q_i,C\otimes_RN)$ is an isomorphism for each integer $i$. Therefore, for all $j\geq 1$, we have:
\begin{align}\tag{\ref{lll}.8}
\Ext^j_R(M^{\triangledown}, N\otimes_RC)  = \H^{j}\left(\Hom_R({\bf Q}, N\otimes_RC) \right)  = \H_{-j-1}({\bf T}\otimes_RN) = \tate_{-j-1}^{\mathcal{G}_C}(M,N).
\end{align}
Fix $t\geq 2$. Then, by the exact sequence in (\ref{at}.1),  we see that $\Ext^{t-1}_R(M^{\triangledown}, N\otimes_RC)\cong\Ext^{t+1}_R(\trk M, N\otimes_RC)$. Consequently, letting $j=t-1$ in (\ref{lll}.8), we  conclude that $\tate_{-t}^{\mathcal{G}_C}(M,N)\cong\Ext^{t-1}_R(M^{\triangledown}, N\otimes_RC)\cong\Ext^{t+1}_R(\trk M, N\otimes_RC)$. Consequently, setting $i=-t$, the conclusion of the proposition follows.
\end{proof}

\begin{lem}\label{ll2} Let $M$ and $N$ be $R$-modules such that $0<\gkd_R(M)<\infty$ and $N\in\ac$. If $\tate_0^{\mathcal{G}_C}(M,N)=0$ and $0\to M\to X\to Y \to 0$ is a $\mathcal{G}_C$-hull of $M$ (see \ref{rm}(ii)),  then $\gmc\Tor_i^R(M,N)\cong\Tor_i^R(X,N)$ for all $i\geq 1$.
\end{lem}

\begin{proof} Note that $Y$ is totally $C$-reflexive, and $\gkd_R(M)=\mathcal{P}_C$-$\pd_R(X)<\infty$ so that $X \in \bc$; see \ref{HD}(ii). It follows by Proposition \ref{l4} that $\tate_i^{\mathcal{G}_C}(X,N)=0$ for all $i\in\ZZ$. Therefore, the $\mathcal{G}_C$-hull of $M$ and Lemma \ref{l3} induce  $\tate_i^{\mathcal{G}_C}(M,N)\cong\tate_{i+1}^{\mathcal{G}_C}(Y,N)$ for all $i\in\ZZ$. Hence, by our assumption,
we have that $\Tor_1^R(Y,N) \cong \tate_1^{\mathcal{G}_C}(Y,N)=0$.

Note that a finite $\mathcal{P}_C$-resolution of $X$ yields an exact sequence of $R$-modules $0\rightarrow X'\rightarrow C^{\oplus t}\rightarrow X\rightarrow0$, where $\pcd_R(X')<\infty$; see \ref{PCR}(iv) and \ref{HD}(ii).

As $N\in \ac$, we have that $\Tor_i^R(C,N)=0$ for all $i\geq 1$. Hence, for all $i\geq 1$, we obtain:
\begin{equation}\tag{\ref{ll2}.1}
\Tor_{i}^R(X',N)\cong\Tor_{i+1}^R(X,N).
\end{equation}

Next, by taking the pull-back of the maps $M\to X$ and $C^{\oplus t}\rightarrow X$, we get the following commutative diagram with exact rows and columns:
$$\begin{CD}
&&&&&&&&\\
	\ \ &&&& 0&&0\\
	&&&& @VVV @VVV\\
	&&&& X' @>{=}>>X'\\
	&&&& @VVV @VVV\\
	\ \ &&0@>>> Z@>>>C^{\oplus t} @>>> Y@>>>0& \\
	&&&& @VVV @VVV @V{\parallel}VV\\
	\ \ &&0@>>> M @>>> X @>>> Y @>>>0&  \\
	&&&& @VVV  @VVV \\
	\ \ &&&& 0&&0\\\\
\end{CD}$$	
As $\pcd_R(X')<\infty$, the exact sequence $0\rightarrow X'\rightarrow Z \rightarrow M\rightarrow0$ is $\gmc$-proper; see \ref{PCR}(i) and Remark \ref{l}(i). Therefore, for each $i\geq 0$, Lemma \ref{ll} yields the following long exact sequence:
\begin{equation}\tag{\ref{ll2}.2}
\cdots\to\gmc\Tor_{i+1}^R(M,N)\to\gmc\Tor_{i}^R(X',N)\to\gmc\Tor_{i}^R(Z,N)\to\cdots
\end{equation}

Note that, $Z$ is totally $C$-reflexive since $Y$ is totally $C$-reflexive. Thus it follows that $\gmc\Tor_i^R(Z,N)=0$ for all $i\geq 1$. Therefore, (\ref{ll2}.2) yields the following isomorphism for all $i\geq 1$:
\begin{equation}\tag{\ref{ll2}.3}
\gmc\Tor_{i+1}^R(M,N)\cong\gmc\Tor_{i}^R(X',N).
\end{equation}
Note that, by Remark \ref{l}(iii), we have $\gmc\Tor_{i}^R(X',N)\cong \Tor_{i}^R(X',N)$ for all $i\geq 1$. Hence, for each $i\geq 1$, (\ref{ll2}.1) and (\ref{ll2}.3) yield:
\begin{equation}\tag{\ref{ll2}.4}
\gmc\Tor_{i+1}^R(M,N)\cong\gmc\Tor_{i}^R(X',N)\cong \Tor_{i}^R(X',N) \cong \Tor_{i+1}^R(X,N).
\end{equation}

Next we consider above pull-back diagram, along with (\ref{ll2}.2), and obtain the following commutative diagram with exact rows:
$$\begin{CD}
	&&&&&&&&\\
	\ \ &&0@>>>\gmc\Tor_1^R(M,N)@>>>X'\otimes_RN @>>> Z\otimes_RN& \\
	&&&& @V{g}VV @V{1}VV @V{f}VV\\
	\ \ &&0@>>>\Tor_1^R(X,N)@>>> X'\otimes_RN  @>>> C^{\oplus t} \otimes_RN& \\\\
	\end{CD}$$
Here the leftmost vertical map $g$ is induced by $1$ and $f$. Recall that $\Tor_1^R(Y,N)=0$. Hence, $f$ is injective. This shows that $g$ is an isomorphism, i.e.,
\begin{equation}\tag{\ref{ll2}.5}
\gmc\Tor_1^R(M,N)\cong\Tor_1^R(X,N).
\end{equation}
Consequently, the conclusion of the lemma follows from (\ref{ll2}.4) and (\ref{ll2}.5).
\end{proof}

The next theorem plays an important role in the proof of Theorem \ref{thmintro2}; it also generalizes Theorem \ref{audepth}, namely Auslander's depth formula, as well as the main result of \cite{CLS}.

\begin{thm}\label{tt} Let $M$ and $N$ be $R$-modules with $\gkd_R(M)<\infty$ and $N\in\ac$. Set $q=\sup\{i\mid\mathcal{G}_C\Tor_i^R(M,N)\neq0\}$. Assume $\tate_i^{\mathcal{G}_C}(M,N)=0$ for all $i\leq0$. If $q=0$ or $\depth_R(\mathcal{G}_C\Tor_q^R(M,N))\leq1$, then it follows that $$\depth_R(M)+\depth_R(N)=\depth R+\depth_R(\gmc\Tor_q^R(M,N))-q.$$
\end{thm}

\begin{proof} Let $0 \to M\to X\to Y \to 0$ be a $\mathcal{G}_C$-hull of $M$; see Remark \ref{rm}(ii). Then $Y$ is totally $C$-reflexive and $\gkd_R(M)=\mathcal{P}_C$-$\pd_R(X)<\infty$ so that $X \in \bc$; see \ref{HD}(ii). Setting $A=\Hom_R(C,X)$ and $B=C\otimes_RN$, we have:
\begin{equation}\tag{\ref{tt}.1}
\depth_R(X)=\depth_R(A)  \text{ and } \depth_R(B)=\depth_R(N).
\end{equation}
The first and second equalities in (\ref{tt}.1) are due to \ref{ABC}(v) and \ref{ABC}(vi), respectively. Also, it follows from \ref{HD}(ii)  that $\pd_R(A)<\infty$. Next we use \ref{TPC}(i) and \ref{TPC}(ii), and record that the following isomorphims hold for all $i\geq 0$:
\begin{equation}\tag{\ref{tt}.2}
\Tor_i^R(A,B) = \Tor_i^R(\Hom_R(C,X),B) \cong \Tor_i^{\mathcal{P}_C}(X,N) \cong \Tor_i^R(X,N).
\end{equation}

We proceed by considering the cases where $q\neq 0$ and $q=0$, separately.

\emph{Case 1.} Assume $q\neq 0$ and $\depth_R(\mathcal{G}_C\Tor_q^R(M,N))\leq1$.\\
Note, as $q\neq 0$, it follows $\gkd_R(M)>0$. So, by Lemma \ref{ll2}, we have that $\gmc\Tor_i^R(M,N) \cong \Tor_i^R(X,N)$ for all $i\geq 1$. Thus (\ref{tt}.2) yields $\Tor_i^R(A,B) \cong \gmc\Tor_i^R(M,N)$ for all $i\geq 1$. This yields $q=\sup\{i\mid\Tor_i^R(A,B)\neq0\}$ and $\depth_R(\Tor_q^R(A,B))=\depth_R(\mathcal{G}_C\Tor_q^R(M,N))\leq 1$. Hence, as $\pd_R(A)<\infty$, we use Theorem \ref{audepth}, applied to the pair $(A,B)$, and conclude that:
\begin{equation}\tag{\ref{tt}.3}
\depth_R(A)+\depth_R(B)=\depth(R)+\depth_R(\Tor_q^R(A,B))-q.
\end{equation}
Now the required equality follows from (\ref{tt}.1) and (\ref{tt}.3).

\emph{Case 2.} Assume $q=0$, i.e, $\mathcal{G}_C\Tor_i^R(M,N)=0$ for all $i\geq 1$.\\
Set $n=\gkd_R(M)$ and proceed by induction on $n$. First assume $n=0$, i.e., $M$ is totally $C$-reflexive. Then it follows by our assumption and Lemma \ref{lll} that
$0=\tate_i^{\mathcal{G}_C}(M,N)\cong \Ext^{-i+1}_R(\trk M, B)$ for all $i\leq 0$. Hence, by \ref{ABC}(vii), we conclude that $M\otimes_RN\cong\Hom_R(M^\triangledown,B)$.

Notice, since $N\in\ac$, it follows that $\Ext^i_R(C,B)=0$ for all $i\geq 1$; see \ref{ABC}(i)(b). As we know the vanishing of $\Ext^{j}_R(\trk M,B)$ for all $j\geq 1$, we see from (\ref{at}.1) that $\Ext^i_R(M^\triangledown,B)=0$ for all $i\geq 1$. Hence \cite[4.1]{AY} yields $\depth_R(B)=\depth_R(\Hom_R(M^\triangledown,B))$. As $M\otimes_RN\cong\Hom_R(M^\triangledown,B)$, we obtain $\depth_R(B)=\depth_R(M\otimes_RN)$. Recall that $\depth_R(N)=\depth_R(B)$; see (\ref{tt}.1). So we conclude that $\depth_R(N)=\depth_R(M\otimes_RN)$. On the other hand, since $M$ is totally $C$-reflexive, we have that $\depth_R(M)=\depth(R)$; see \ref{AA}(ii). This establishes the equality $\depth_R(M)+\depth_R(N)=\depth(R)+\depth_R(M\otimes_RN)$, and completes the proof for the case where $n=0$.

Next assume $n\geq 1$. It follows from Proposition \ref{l4} that $\tate_i^{\mathcal{G}_C}(X,N)=0$ for all $i\in\ZZ$. Hence, for all $i\in\ZZ$, the exact sequence $0 \to M\to X\to Y \to 0$, i.e., the $\mathcal{G}_C$-hull of $M$, and Lemma \ref{l3} yield the isomorphisms:
\begin{equation}\tag{\ref{tt}.4}
\tate_i^{\mathcal{G}_C}(M,N) \cong \tate_{i+1}^{\mathcal{G}_C}(Y,N).
\end{equation}
So, our assumption on the vanishing of $\tate_i^{\mathcal{G}_C}(M,N)$ and (\ref{tt}.4) imply that $\tate_{j}^{\mathcal{G}_C}(Y,N)=0$ for all $j\leq 1$. As $Y$ is totally $C$-reflexive, it follows that $\depth_R(Y)=\depth(R)$ and $\mathcal{G}_C\Tor_i^R(Y,N)=0$ for all $i\geq 1$. Therefore, by the induction hypothesis, we have:
\begin{equation}\tag{\ref{tt}.5}
\depth_R(Y\otimes_RN)=\depth_R(N).
\end{equation}

Recall that, by (\ref{tt}.2), we have $\Tor_i^R(A,B) \cong \Tor_i^R(X,N)$ for all $i\geq 0$. This shows $A\otimes_R B \cong X\otimes_RN$. Note, since $n=\gkd_R(M)>0$ and $q=0$, Lemma \ref{ll2} implies that $\Tor_i^R(X,N)=0$ for all $i\geq 1$. Therefore, we see $\Tor_i^R(A,B)=0$ for all $i\geq 1$. Recall $\pd_R(A)<\infty$ and note $n=\pcd_R(X)=\pd_R(A)$; see \ref{HD}(ii). Thus, we use Theorem \ref{audepth} for the pair $(A,B)$ and obtain:
\begin{equation}\tag{\ref{tt}.6}
\depth_R(X\otimes_RN)=\depth_R(A\otimes_RB)=\depth_R(B)-n=\depth_R(N)-n < \depth_R(N).
\end{equation}

We have:
\begin{equation}\tag{\ref{tt}.7}
0=\tate_0^{\mathcal{G}_C}(M,N) \cong \tate_{1}^{\mathcal{G}_C}(Y,N) \cong \Tor_1^R(Y,N).
\end{equation}
In (\ref{tt}.7), the first equality is by the hypothesis, while the second and the third isomorphisms follow from (\ref{tt}.4) and \ref{TTS}, respectively.
So the short exact sequence $0 \to M\to X\to Y \to 0$ induces the following exact sequence:
\begin{equation}\tag{\ref{tt}.8}
0\longrightarrow M\otimes_RN\longrightarrow X\otimes_RN\longrightarrow Y\otimes_RN\longrightarrow0.
\end{equation}

It follows from (\ref{tt}.5) and (\ref{tt}.6) that $\depth_R(X\otimes_RN) < \depth_R(Y\otimes_RN)$. Hence the depth lemma applied to (\ref{tt}.8)
shows that $\depth_R(M\otimes_RN)=\depth_R(X\otimes_RN)$. So we conclude that $\depth_R(M\otimes_RN)=\depth_R(N)-n$; see (\ref{tt}.6). Recall, by \ref{AA}(ii), we have $n=\depth(R)-\depth_R(M)$. Therefore, we obtain the depth equality $\depth_R(M\otimes_RN)=\depth_R(N)-n=\depth_R(M)+\depth_R(N)-\depth(R)$, which completes the proof.
\end{proof}

It is known that, if $M$ and $N$ are $R$-modules such that $\G-dim_R(M)<\infty$ and $\id_R(N)<\infty$, then Tate Tor modules $\tate_i^{R}(M,N)$ vanish for all $i\in\ZZ$; see \cite[2.4(1)]{CJ}. In the following proposition we show that this vanishing result carries over to Tate Tor modules with respect to a semidualizing module $C$; see \ref{TTS}. Recall that, if $N$ is an $R$-module such that $\mathcal{I}_{C}$-$\id_R(N)<\infty$, then $N\in\mathcal{A}_C$; see \ref{HD}(iii).

\begin{prop}\label{llll}  Let $M$ and $N$ be $R$-modules such that $\gkd_R(M)<\infty$ and $\mathcal{I}_{C}$-$\id_R(N)<\infty$. Then it follows that $\tate_i^{\mathcal{G}_C}(M,N)=0$ for all $i\in\ZZ$. In particular, we have $\Tor_i^R (M,N)=0$ for all $i>\gkd_R(M)$.
\end{prop}

\begin{proof} Note that $N\in \ac$; see \ref{HD}(iii). As the second part of the claim follows from the construction of Tate Tor modules, we proceed and prove the vanishing of $\tate_i^{\mathcal{G}_C}(M,N)=0$ for all $i\in\ZZ$; see \ref{TTS}. 

Claim 1: If $F$ is a free $R$-module, then $\tate_j^{\mathcal{G}_C}(F,N)=0$ for all $j\in\ZZ$.\\
Proof of Claim 1:  Let $F$ be a free $R$-module. Then $\tate_j^{\mathcal{G}_C}(F,N) \cong \Tor_j^R(F,N)=0$ for all $j>\gkd_R(F)=0$; see \ref{TTS}. Moreover, $\Ext^{j}_R(\trk F,-)=0$ for all $j\geq 1$ since $\trk F =0$. Hence Claim 1 follows from Proposition \ref{lll}.

Claim 2: If $T$ is a totally $C$-reflexive $R$-module, then it follows that $\tate_i^{\mathcal{G}_C}(T,N)=0$ for all $i\leq 0$.\\
Proof of Claim 2: Let $T$ be a totally $C$-reflexive $R$-module. Then it follows that $\depth_R(\trk T)=\depth(R)$; see \ref{AA}(i). Note that $\id_R(N\otimes_RC)<\infty$; see \ref{HD}(iii). Therefore, we conclude from \cite[2.6]{FI} that $\Ext^{j}_R(\trk T,N\otimes_RC)=0$ for all $j\geq 1$. Now Proposition \ref{lll} shows that $\tate_i^{\mathcal{G}_C}(T,N)=0$ for all $i\leq 0$, i.e., the claim follows.

Next consider a $\mathcal{G}_C$-approximation of $M$, i.e., a short exact sequence $0\to Y\to X\to M \to 0$, where $X$ is a totally $C$-reflexive $R$-module and $\pcd_R(Y)<\infty$; see Remark \ref{rm}(i). As $N\in \ac$, it follows from Proposition \ref{l4} that $\tate_i^{\mathcal{G}_C}(Y,N)=0$ for all $i\in\ZZ$. Therefore,  Lemma \ref{l3} yields $\tate_i^{\mathcal{G}_C}(M,N)\cong\tate_i^{\mathcal{G}_C}(X,N)$ for all $i\in\ZZ$. So, in view of Claim 2, it is enough to prove the vanishing of $\tate_i^{\mathcal{G}_C}(X,N)$ for all $i\geq 1$.

Let $n$ be a positive integer and consider $\tate_n^{\mathcal{G}_C}(X,N)$. Note that the conclusion of Claim 1 allows us to use the dimension shifting and conclude that $\tate_n^{\mathcal{G}_C}(X,N)\cong \tate_0^{\mathcal{G}_C}(\Omega^{n} X,N)$, where $\Omega^{n} X$ is the $n$th syzygy of $X$. We know $X$ is totally $C$-reflexive; so $\Omega^{n} X$ is also totally $C$-reflexive; see \ref{AA}(ii). Hence the vanishing of $\tate_0^{\mathcal{G}_C}(\Omega^{n} X,N)$ follows from Claim 2. This argument implies that $\tate_i^{\mathcal{G}_C}(X,N)=0$ for all $i\geq 1$ and completes the proof.
\end{proof}

We are now ready to give a proof of Theorem \ref{thmintro2}:

\begin{proof}[Proof of Theorem \ref{thmintro2}] Note that, by Proposition \ref{llll}, we have $\tate_i^{\mathcal{G}_C}(M,N)=0$ for all $i\in\ZZ$. Note also that $N$ belongs to the Auslander class $\mathcal{A}_C$; see \ref{HD}(iii). Therefore, Theorem \ref{thmintro} implies that $\Tor_i^R(M,N) \cong \mathcal{G}_C\Tor^R_i(M,N)$ for all $i\geq 0$. Now the required conclusion follows from Theorem \ref{tt}.
\end{proof}

\section*{Appendix: A discussion on Auslander's depth formula  }\label{sec:appendix}
\noindent
In this section, we revisit Theorem \ref{audepth}, namely Auslander's the depth formula, discuss some of its applications, and highlight its relation with some open problems in the literature. We should note that the depth formula, especially for the case where $q=0$, i.e., for the case where the modules in question are Tor-independent, is an important tool in commutative homological algebra and has many applications in the literature; see, for example, \cite{CI, CJ}. Here we discuss only a few basic applications and consequences of the depth formula. The main purpose of such a discussion is to provide some motivation for our work from section 4 concerning the generalizations of the depth formula, explain why the depth formula is useful, and indicate why it is worth studying the conditions implying the depth formula to hold.  In the following, we record various remarks about the depth formula:
\begin{enumerate}[\rm(1)]
\item It is an open question whether or not the depth formula always holds for Tor-independent modules. More precisely, it is not known if the equality $\depth_R(M)+\depth_R(N)=\depth(R)+\depth_R(M\otimes_RN)$ holds whenever $M$ and $N$ are nonzero $R$-modules such that $\Tor_i^R(M,N)=0$ for all $\geq 1$.
\item If $M$ and $N$ are $R$-modules, then the depth equality $\depth_R(M)+\depth_R(N)=\depth(R)+\depth_R(M\otimes_RN)$ may fail, in general, even if $\Tor_i^R(M,N)=0$ for all $i\gg 0$. For example, if $R=k[\![x,y,z]\!]/(xz-y^2,xy-z^2)$, $M=R/(x,y)$ and $N=R/(x,z)$, then it follows that $\Tor_1^R(M,N)\neq 0$ and $\Tor_i^R(M,N)=0$ for all $i\geq 2$; see \cite[4.2]{J1}. Furthermore, one has  $\depth_R(M)+\depth_R(N)=0<1=\depth(R)+\depth_R(M\otimes_RN)$. On the other hand, the depth formula in the form $\depth_R(M)+\depth_R(N)=\depth R+\depth_R(\Tor_q^R(M,N))-q$ holds, where $q=\sup\{i\mid\Tor_i^R(M,N)\neq0\}=1$.
\item The depth formula is an extension of the classical Auslander-Buchsbaum formula. In fact, let $M$ be a nonzero $R$-module such that $\pd_R(M)<\infty$ and let $N=k$. Then it follows from Theorem \ref{audepth} that $\depth_R(M)+\depth_R(N)=\depth R+\depth_R(\Tor_q^R(M,N))-q$, where $q=\pd_R(M)$. This yields the classical Auslander-Buchsbaum formula, i.e., $\pd_R(M)+\depth_R(M)=\depth(R)$; see also Jorgensen's dependency formula  \cite{J2} for further applications of the depth formula in this direction.
\item The depth formula determines the projective dimension of the tensor product of two nonzero Tor-independent modules, when each of the module considered has finite projective dimension. More precisely, let $M$ and $N$ be nonzero $R$-modules such that $\pd_R(M)<\infty$ and $\pd_R(N)<\infty$.
If $\Tor_i^R(M,N)=0$ for all $i\geq 1$, then it follows that $\pd_R(M\otimes_RN)<\infty$, and hence the depth formula yields
$$(\depth(R)-\pd_R(M))+(\depth(R)-\pd_R(N))=\depth(R)+(\depth(R)-\pd_R(M\otimes_RN)),$$
which recovers the well-known equality $\pd_R(M\otimes_RN)=\pd_R(M)+\pd_R(N)$.
\item If $M$ and $N$ are nonzero $R$-modules such that $M\otimes_RN$ is an $n$th syzygy module, then the depth formula can be used to prove that $M$ or $N$ are $n$th syzygy modules. For example, Huneke and Wiegand \cite[2.6]{HW} used the depth formula and proved that, if $R$ is a complete intersection domain and $M\otimes_RN$ is an $n$th syzygy module, then both $M$ and $N$ are $n$th sygzygy modules; see also \cite{CeSa} for several applications of the depth formula on the depth of tensor products of modules over complete intersection rings.
\item The depth formula is useful in studying the Cohen-Macaulayness of tensor products. For example, it can be used to prove the following result of Kawasaki \cite[3.3(i)]{Kawa}: if $R$ is a Cohen-Macaulay local ring with canonical module $\omega$ and $M$ is a nonzero $R$-module such that $\pd_R(M)<\infty$, then $M\otimes_{R}\omega$ is Cohen-Macaulay if and only if $M$ is Cohen-Macaulay. In fact, if $N$ is a faithful maximal Cohen-Macaulay $R$-module (e.g., $N=\omega$) and $M$ is a nonzero $R$-module such that $\pd_R(M)<\infty$, then $M$ and $N$ are Tor-independent \cite[2.2]{Yos} so that the depth formula yields $\depth_R(M)=\depth_R(M\otimes_RN)\leq \dim_R(M\otimes_RN)=\dim_R(M)$, which shows $M$ is Cohen-Macaulay if and only if $M\otimes_RN$ is Cohen-Macaulay; see \cite{KS} for further applications of the depth formula in this direction.
\item It is an open question whether or not the vanishing of all higher homology modules imply finite homological dimension. More precisely, if $M$ is an $R$-module such that $\Tor_i^R(M,M)=0$ for all $i \geq 1$, then it is not known whether or not $M$ must have finite projective dimension, in general; see, for example, \cite[the paragraph preceeding 2.6]{Sega}. It is also not known if this question has an affirmative answer in case the depth formula for Tor-independent $R$-modules always holds. It seems worth pointing out that easy applications of the depth formula give affirmative answers for this Tor-vanishing problem in some special, albeit important, cases.

For example, if $R$ is a one-dimensional domain over which the depth formula holds for all Tor-independent $R$-modules, and $I$ is an ideal of $R$ such that $\Tor_i^R(I,I)=0$ for all $i \geq 1$, then the depth formula yields the equality $\depth_R(I)+\depth_R(I)=\depth(R)+\depth_R(I\otimes_RI)$, which implies that $I\otimes_RI$ is torsion-free. This shows that $I$ is a free $R$-module; see, for example, \cite[the paragraph preceding 4.4]{HW}.

Another application of the depth formula, which is related to the aforementioned Tor-vanishing problem is that, if $R=S/(x)$, where $S$ is a reduced Cohen-Macaulay ring and $(x,y)$ is a pair of exact zero-divisors such that $yM=0$, then the vanishing of $\Tor_i^R(M,M)$ for all $i\geq 1$ implies that $M$ has infinite complete intersection dimension, and hence infinite projective dimension; see \cite[2.10]{BCJ}.

In view of the foregoing discussion, it seems reasonable to raise the following question: if $M$ is an $R$-module such that $\Tor_i^R(M,M)=0$ for all $i \geq 1$, and the depth formula holds for all Tor-independent $R$-modules, then must $\pd_R(M)<\infty$? This problem can also be considered for different types of Tors, including Tate homology with respect to a semidualizing module $C$, or $\gmc$-relative homology; see \ref{RTS} and \ref{TTS}.

\item A long standing conjecture of Huneke and Wiegand \cite[page 473]{HW} claims that, if $R$ is a one-dimensional domain and $M$ is an $R$-module such that $M\otimes_RM^{\ast}$ is a nonzero and torsion-free, where $M^{\ast}=\Hom_R(M,R)$, then $M$ is free. This conjecture is wide open, even for two-generated ideals over complete intersection rings. It seems worth noting that an easy application of the depth formula give an affirmative answer
to the conjecture when the module considered has finite projective dimension: if $\pd_R(M)<\infty$, then $\Tor_i^R(M,M^{\ast})=0$ for all $i \geq 1$ \cite[2.2]{Yos} so that the depth formula shows that $M$ is torsion-free; consequently, $M$ is free by the Auslander-Buchsbaum formula. In fact, by using the depth formula, one can easily see that, if $R$ is a one-dimensional ring, and $M$ and $N$ are $R$-modules such that $M\otimes_RN$ is torsion-free and $\pd_R(M)<\infty$, then $M$ is free.
\item There are several conjectures from the representation theory of finite-dimensional algebras that have been transplanted to commutative algebra. An example is the celebrated {\it Auslander-Reiten Conjecture}, which states that a finitely generated module $M$ over a finite-dimensional algebra $A$ satisfying $\Ext^{i}_{A}(M, M)=\Ext^{i}_{A}(M, A)=0$ for all $i\geq 1$ must be projective. This long-standing conjecture of Auslander and Reiten is closely related to other important conjectures such as the Finitistic Dimension Conjecture from representation theory.

The Auslander-Reiten conjecture, over Gorenstein rings, can be stated as follows: if $M$ is a maximal Cohen-Macaulay $R$-module such that $\Ext^{i}_{R}(M, M)=0$ for all $i\geq 1$, then $M$ must be free.
If, in addition, the depth formula holds for Tor-independent modules over the Gorenstein ring considered, then the conjecture is equivalent to the following claim: if $M$ or $M\otimes_RM^{\ast}$ is a maximal Cohen-Macaulay $R$-module such that $\Tor_i^R(M,M^{\ast})=0$ for all $i\geq 1$, then $M$ must be free; see, for example, \cite[2.7]{J3}. Thus, the depth formula, when holds, gives a new perspective on the celebrated conjecture of Auslander and Reiten, and is one of the tools that allows one to consider the conjecture in terms of the vanishing of homology and the depth of the tensor product of modules.
\end{enumerate}

\section*{Acknowledgements}
\noindent
We thank the anonymous referees for many useful comments that helped improve the exposition.

\end{document}